\documentclass[12pt]{amsart}
\usepackage{amssymb}
\usepackage{bm}

\usepackage[
width=31pc,
height=48pc,
margin=1in,
footskip=30pt,
]{geometry}
\usepackage{layout}

\usepackage[numbers]{natbib}

\usepackage{graphicx}
\usepackage{epsfig}
\usepackage{subcaption}
\usepackage{amsmath,amsfonts,epstopdf}
\usepackage{tikz-cd}
\usepackage{multirow}
\usepackage{enumitem}
\usepackage{mathtools}

\usepackage{hyperref}
\usepackage[notref,notcite,color]{showkeys}

\newcommand{\alignedintertext}[1]{%
  \noalign{%
    \vskip\belowdisplayshortskip
    \vtop{\hsize=\linewidth#1\par
    \expandafter}%
    \expandafter\prevdepth\the\prevdepth
  }%
}

\usepackage{verbatim}

\graphicspath{{figures/}}

\theoremstyle{plain}
\newtheorem{theorem}{Theorem}[section]

\newtheorem{corollary}[theorem]{Corollary}
\theoremstyle{definition}

\newtheorem{remark}[theorem]{Remark}

\numberwithin{equation}{section}

\newcommand{\R}{\mathbb{R}}
\newcommand{\N}{\mathbb{N}}
\newcommand{\Z}{\mathbb{Z}}

\newcommand{\C}{\mathbb{C}}

\newcommand{\Q}{\mathbb{Q}}
\newcommand{\bbB}{\mathbb{B}}
\newcommand{\bbC}{\mathbb{C}}
\newcommand{\bbH}{\mathbb{H}}

\newcommand{\ol}{\overline}

\newcommand{\cali}{\mathcal{I}}

\newcommand{\calp}{\mathcal{P}}

\renewcommand{\Im}{\mathrm{Im}}
\renewcommand{\Re}{\mathrm{Re}}

\title{Construction of invariant curves for some Piecewise Isometries}

\author[Noah Cockram, Peter Ashwin and Ana Rodrigues]{Noah Cockram$^1$, Peter Ashwin$^1$ and Ana Rodrigues$^{1,2,3}$}

\address{$^1$ Department of Mathematics and Statistics, University of Exeter, Exeter EX4 4QE, UK} 

\address{$^2$ Departamento de Matem\'{a}tica, Escola de Ci\^{e}ncias  e Tecnologia, Universidade de \'{E}vora, Rua Rom\~{a}o Ramalho, 59, 7000--671 \'{E}vora, Portugal }

\address{$^3$ Centro de Investiga\c{c}\~{a}o em Matem\'{a}tica e Aplica\c{c}\~{o}es, Rua Rom\~{a}o  Ramalho, 59, 7000--671 \'{E}vora, Portugal} 

\begin{document}

\begin{abstract}
We establish conditions for the existence of a family of piecewise linear invariant curves in a two-parameter family of piecewise isometries on the upper half-plane known as Translated Cone Exchange Transformations. We show that these curves are embeddings of interval exchange transformations and give rise to layers of invariant regions. 
We also show the existence of a trapezoidal piecewise isometry for which the dynamics on the top and bottom edges are distinct 2-interval exchange transformations.
\end{abstract}

\maketitle

\section{Introduction}

An \textit{interval exchange transformation} (IET) is a pair $(\mathcal{I}, f)$ such that $f$ is a bijective map that is a translation on each set in a partition $\mathcal{I}$ of an interval $I$ into intervals. If the partition consists of $d$ then we say it is a $d$-IET. Such a map can be parametrized in terms of $\alpha\in\R_+^d$ and a permutation $\sigma\in S(\{0,\ldots,d-1\})$. For some $x_0$ we set $x_j=x_0+\sum_{k=1}^j \alpha_k$ and $f$ is defined in terms of a partition into intervals $I_j\supset (x_j,x_{j+1})$  by translations
$$
f_j(x)=x+\tau_j
$$
for $x\in I_j$. The $\tau_j:=\sum_{k: \sigma(k)<\sigma(j)}\alpha_k-\sum_{k<j} \alpha_k$, are such that the intervals are reassembled in an order defined by the permutation $\sigma$. Note that the set of pre-images and images of the discontinuities will be countable and so the value of $f$ on the discontinuities can be ignored if the aim is to characterise a full Lebesgue measure subset of $I$.

There is a deep theory for the dynamics of iterated IETs; see for instance \cite{KA, Ke, V1, V4, AG1}. It is known that almost all IETs which are not irrational rotations are weakly mixing \cite{AG1}, though not (strongly-)mixing \cite{KA}. They give insight into many related concepts, such as Lagrange's Theorem \cite{BC}, continued fractions and the Gauss map \cite{Z1}, Khinchin's Theorem \cite{LM} and Sturmian shifts \cite{FZ1}.  IETs have deep connections in the study of measurable foliations and translation surfaces \cite{M1}. 

\textit{Piecewise isometries (PWIs)} represent a generalisation of IETs to higher dimensions and arbitrary metric spaces. They consist of $(\mathcal{P},F)$ where $\mathcal{P}$ is a partition of some metric space $P$ into convex pieces $P_j$, and $F$ is a map that maps the pieces through isometries onto its image \cite{Ash02,Goetz1,Goetz2}: on the complex plane this means that, for example, if $z\in P_j$ then we have
$$
F_j(z)= z\exp(i \tau_j)+\beta_j.
$$
Even for such orientation-preserving cases, the dynamics of PWIs are much less understood than IETs. For PWIs that are not bijections, one can reduce to a bijection (almost everywhere) by taking $P$ the maximal invariant set $P_{\max}=\ol{\cap_{k>0} F^k(P)}$. This maximal invariant set can be divided into the exceptional set $\mathcal{E}$ - consisting of the orbit of the discontinuity and the orbits which accumulate on it - and its complement, the regular set $\mathcal{R}$ - consisting of a packing of convex periodic islands.

It is known that PWIs have zero topological entropy \cite{Buz} and there are conjectured conditions for a PWI to have sensitive dependence on initial conditions in the exceptional set \cite{BK}. Mostly it has been specific examples that have been studied \cite{GA, GP, AKT, LKV, AG06} and PWIs are known to exhibit a seemingly wide variety of behaviour not seen with IETs, such as the appearance of unbounded periodicity which accumulates on the exceptional set. The example given in \cite{AKT} shows that the exceptional set can be a Cantor set on which the dynamics is minimal and uniquely ergodic. 



PWIs have been studied as linear models for the standard map (e.g. \cite{Ash97}), and are known to exhibit similar phenomena in terms of regions of regular and chaotic motion. Unlike IETs which are typically ergodic, there is plenty of evidence, for example in \cite{AG06a}, that the exceptional set may not have dense trajectories for many families of PWIs. One of the obstructions to ergodicity in the exceptional set is the existence of invariant curves that prevent orbits from crossing. Similarly, \cite{AG06} presents a PWI on the plane which consists of a permutation of four cones by rational rotations, which was shown by the authors to admit an uncountable number of closed, piecewise linear curves on which the dynamics is conjugate to a transitive IET.


It is clear from \cite{AGPR} that both trivial invariant curves (consisting of unions of line segments or circle arcs) and nontrivial invariant curves (that are not of this form) can appear as embeddings of IETs in PWIs. In \cite{AGPR} it is shown that there are no non-trivial embeddings of 2-IETs into 2-PWIs, but and that there is a 3-PWI that admits a non-trivial embedding of any 3-IET. 
Beyond these results, there is very little known about non-trivial invariant curves in piecewise isometries, and even the seemingly simpler ``trivial'' variety of invariant curves (those which consist of arcs and line segments) have up until now have been studied very little. 

In this paper we study two particular cases of PWIs: in the first, we exactly describe the maximal invariant set and the dynamics on it; in the second case, we produce a sequence of invariant curves that gives useful bounds on the maximal invariant set. 
Our main results are as follows.
In Theorem \ref{InvSetThm}, we provide an exact description of the maximal invariant set and global attractor for a class of PWIs.
In Theorem \ref{invariantcurves}, we prove the existence of a sequence of 2-parameter families of piecewise linear (polygonal) invariant curves for a class of PWIs, which are continuous in both parameters; these curves are embeddings of IETs, and also form part of the boundary of an invariant region.
In Theorem \ref{densitythm}, we establish sufficient and necessary conditions for the density (or periodicity everywhere) of orbits on the aforementioned invariant curves due to the underlying IET.

The paper is organized as follows. In Section 2, we introduce definitions and terminology relevant to the results of this paper.  In Section 3, we investigate a class of PWIs with three unbounded atoms, and state and prove Theorem \ref{InvSetThm} as well as describe the dynamics on the maximal invariant set.  In section 4, we study a class of PWIs with four unbounded atoms. We state and prove Theorem \ref{invariantcurves} and Theorem \ref{densitythm}.  In section 5, we give concluding remarks and discuss some open questions.

\section{Interval Exchange Transformations and Translated Cone Exchanges}

We consider a particular class of PWIs, Translated Cone Exchange Transformations, introduced in \cite{AGPR}.  These are described by a partition of the upper half of the complex plane divided into cones $P_0$, $P_1$, ..., $P_{d+1}$ (for some $d > 2$) which share a common vertex at the origin.  The cones are permuted via rotation and translation while the end cones are only translated.

Let $\mathbb{H}= \{z \in \C: \Im(z) > 0\} \subset \C$ denote the upper half plane and let $\mathbb{B}^{d+2}$ be the $(d+1)$-simplex:
\begin{equation*}
\mathbb{B}^{d+2} = \left\{ \alpha = (\alpha_0, \ldots, \alpha_{d+1}) \in (0, \pi)^{d+2} : \sum_{j=0}^{d+1} \alpha_j = \pi \right\}.
\end{equation*}
Given $\alpha = (\alpha_0, \ldots, \alpha_{d+1}) \in \mathbb{B}^{d+2}$ we partition the interval $J = [0,\pi]$ into subintervals
\begin{equation*}
\begin{cases}
		J_0\supset (0, \alpha_0),												& \text{ if } j = 0, \\[1.5mm]
		J_j\supset \left( \sum \limits_{k=0}^{j-1} \alpha_k, \sum_{k=0}^j \alpha_k \right),	& \text{ if } j \in \{1,\ldots,d+1\}
		\end{cases}
\end{equation*}
For such a $\alpha\in\bbB^{d+2}$ and permutation $\sigma\in S(\{0, \ldots, d+1\})$ there is a $d+2$-IET $g$ with partition $\{J_j\}_{j=0}^{d+1}$ for $j=0,\ldots,d+1$ that acts on $[0,\pi]$; we write $g_j(x)=x+\tau_j$ if $x\in J_j$ for $j=0..d+1$. The partition $\{J_j\}$ induces a partition of $\bbH$ into cones
\begin{equation}
P_j\supset\{z\in \bbH~:~\arg(z)\in J_j\}.
\end{equation}
for $j=0..d+1$ (for definiteness we assign $z=0$ to lie in $P_0$).

Given $\alpha\in\bbB^{d+2}$ and permutation $\sigma\in S(\{0, \ldots, d+1\})$ and $(\beta_0,\cdots,\beta_{d+1})\in\bbC^{d+2}$ we define similarly to \cite{AGPR} a general {\em translated cone exchange} (TCE) with this data as a map
$F:\ol{\bbH}\rightarrow \C$ such that
$$
F_j(z)=  |z| \exp (ig(\arg(z))) +\beta_j =  z \exp (i\tau_j) +\beta_j
$$
if $z\in P_j$. Figure~\ref{fig:P} shows an example for $d=4$ such a partition into such as set of cones $\calp$
\begin{equation*}
\calp = \left\{ P_j : j \in \{0, \ldots, d+1\} \right\},
\end{equation*}
We consider, as in \cite{AGPR}, TCEs such that $\sigma(0)=0$ and $\sigma(d+1)=d+1$ meaning the permutation only affects the cones in $P_1,\cdots P_d$, and we take the particular set of translations
$$
\begin{aligned}
\beta_0=& -1, \\ 
\beta_j=&-\eta,~~\mbox{ if }j=1,\ldots,d\\
\beta_{d+1}& = \lambda
\end{aligned}
$$
for $0<\lambda<1$ real, where $\eta=1-\lambda$.

\begin{figure}
\centering
\includegraphics[width=0.7\linewidth]{{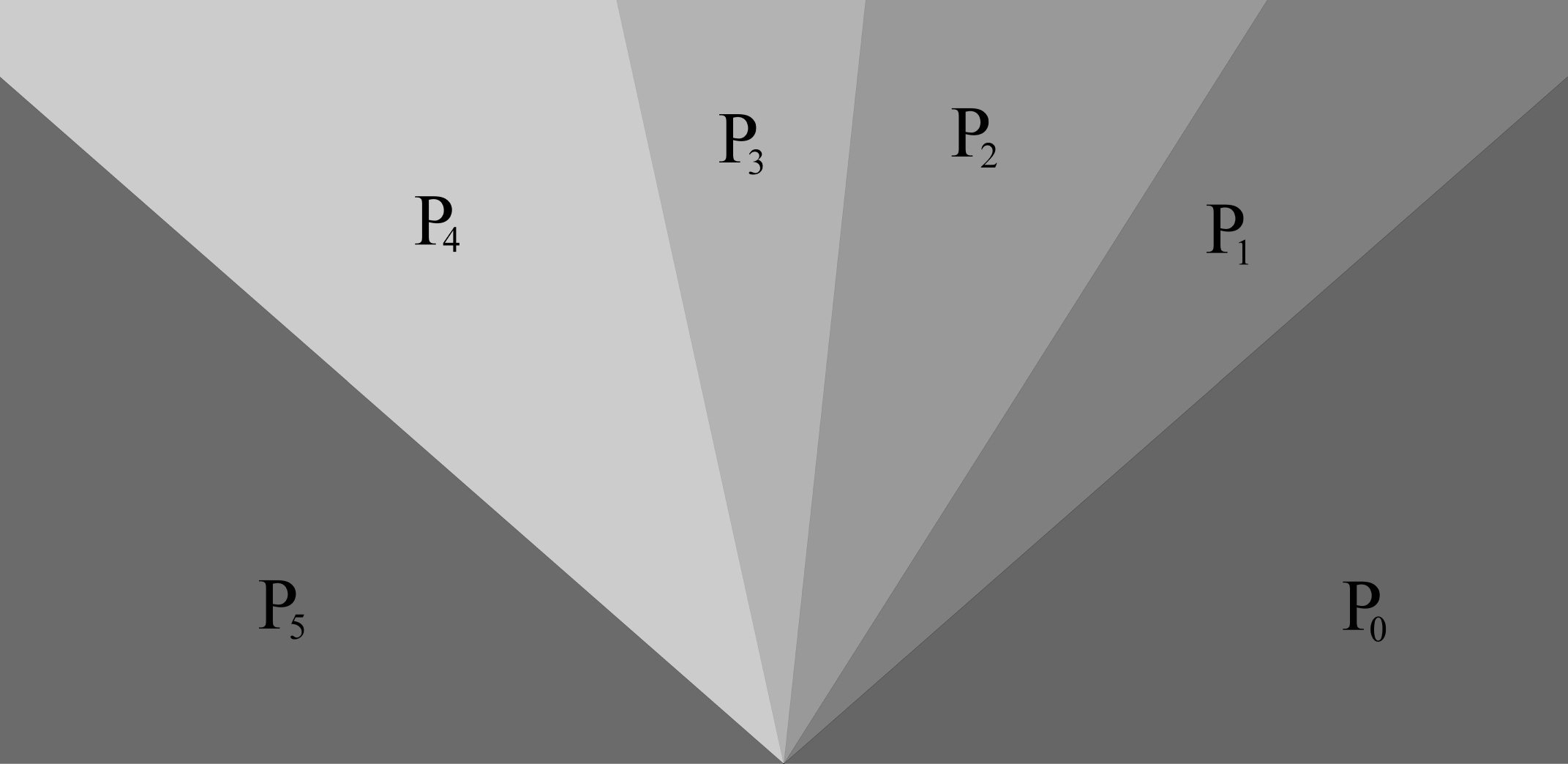}}

(a)

~

\includegraphics[width=0.7\linewidth]{{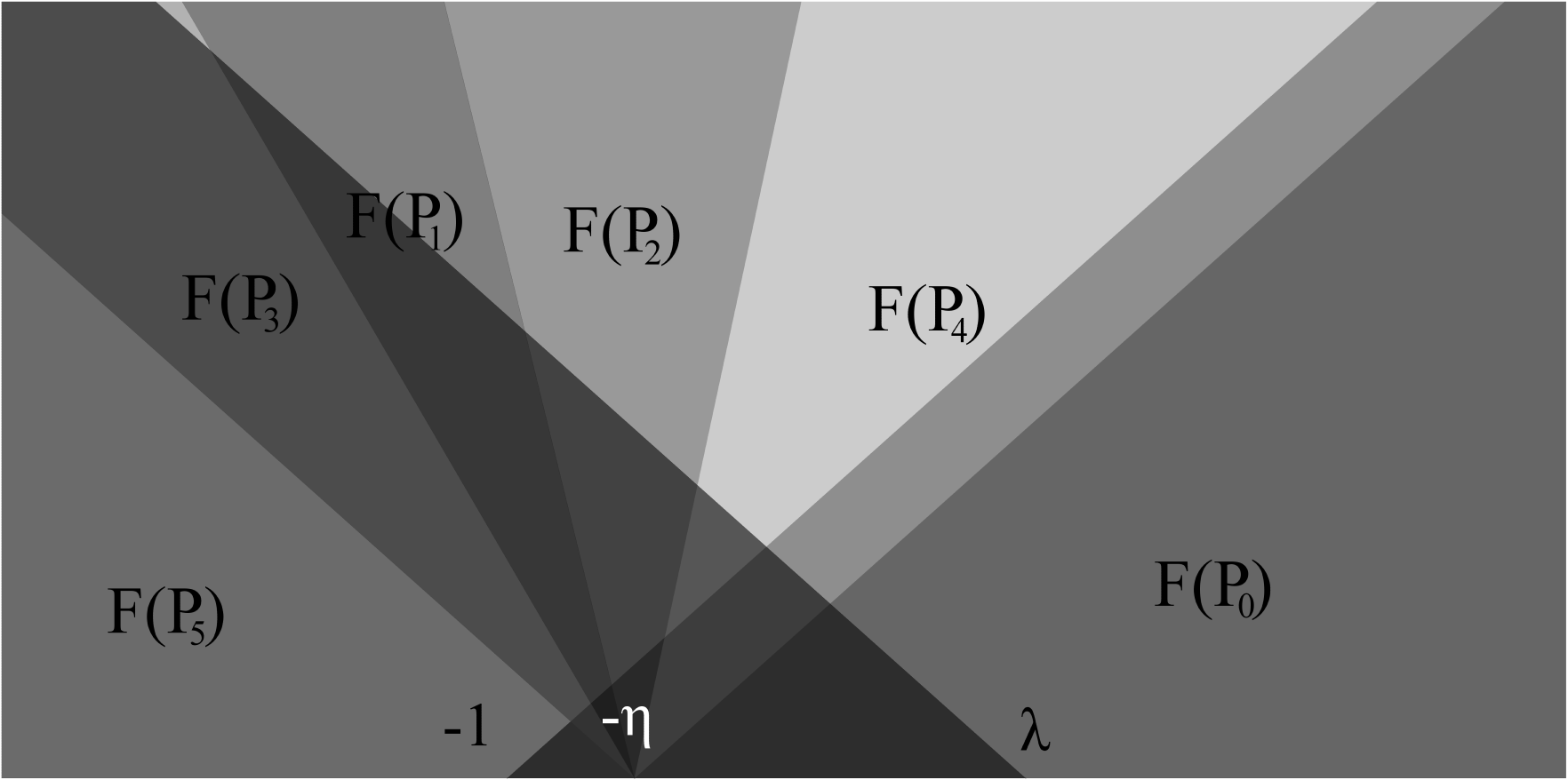}}

(b)

\caption{(a): An example of a partition $\calp$ of the closed upper half plane $\overline{\mathbb{H}}$ into $d+2=6$ cones. (b) The partition $\calp$ is transformed by a translated cone exchange $F$.  Note firstly that the $d=4$ cones $P_1$, $P_2$, $P_3$ and $P_4$ have been permuted via rotation about their shared vertex and translated together while the outer two cones have been translated to overlap the cone vertex.}
\label{fig:P}
\end{figure}

In summary, we consider in this paper a family of TCEs parametrized by $0<\lambda<1$, $\alpha\in \bbB^{d+1}$ and $\sigma\in S(\{0,\ldots,d+1\})$ where we require $\sigma(0)=0$ and $\sigma(d+1)=d+1$. For this data we define $F : \ol{\mathbb{H}} \rightarrow \ol{\mathbb{H}}$ by
\begin{equation}
\label{eq:F}
F(z) =   \begin{cases}
                z -1      &\text{if $x \in P_0$}, \\
                z\exp(i\tau_j) -\eta    &\text{if $z \in P_j, j \in \{1,\ldots,d\}$}, \\
                z + \lambda    &\text{if $z \in P_{d+1}$},
                \end{cases}
\end{equation}
where $\eta=1-\lambda$ and
\begin{equation}
\label{eq:tau}
\tau_j = \sum_{\sigma(k) < \sigma(j)} \alpha_k - \sum_{k < j} \alpha_k.
\end{equation}
This map is such that there is a $2$-IET on the interval $[-1,\lambda]$ that permutes the intervals $[-1,0]$ and $[0,\lambda]$ on
in the real axis within $\ol{\bbH}$ and such that for some region in the upper half plane, above this interval, the map is invertible: see Figure~\ref{fig:P}. 

There are many unanswered questions even for this quite restricted family of TCEs \cite{AGPR}. For example, it is conjectured that this family of TCEs will have a maximal invariant set with positive Lebesgue measure. The specific question we address here is whether the maximal invariant set contains a neighbourhood of the baseline in $\ol{\bbH}$.

Now consider any $k$-IET $(\cali,f)$ on $I$ with data $\alpha'$ and $\sigma'$. If there is a map $\gamma:I\rightarrow \ol{\bbH}$ that is a homeomorphism onto its image such that
\begin{equation}
F\circ \gamma (x) =\gamma \circ f(x)
\end{equation}
except on a countable set of $x$, then we say (as in \cite{AGPR}) that $\gamma(I)$ is a {\em continuous embedding} of the IET $f:I\rightarrow I$ into $F:\ol{\bbH}\rightarrow \ol{\bbH}$. We say an embedding $\gamma$ of an IET into a PWI is a \textit{linear embedding} if $\gamma(I)$ is a union of lines, and it is a \textit{fan embedding} if $\gamma(I)$ is a union of segments of circles. Recall from \cite{AGPR} that we call $\gamma$ a trivial embedding if it is a linear or a fan embedding, and that paper demonstrated that there can be nontrivial embeddings.

\section{A TCE interpolating between IETs}

We start by considering a simple but nontrivial case where $d=1$, and we consider the TCE (\ref{eq:F}) for parameters
\begin{equation}
\label{eq:parameters1}
\alpha \in \mathbb{B}^3, ~0<\lambda<1, ~\text{and } \sigma \text{~the identity permutation}.
\end{equation}

\begin{figure}
    \centering
    \includegraphics[width=0.8\linewidth]{{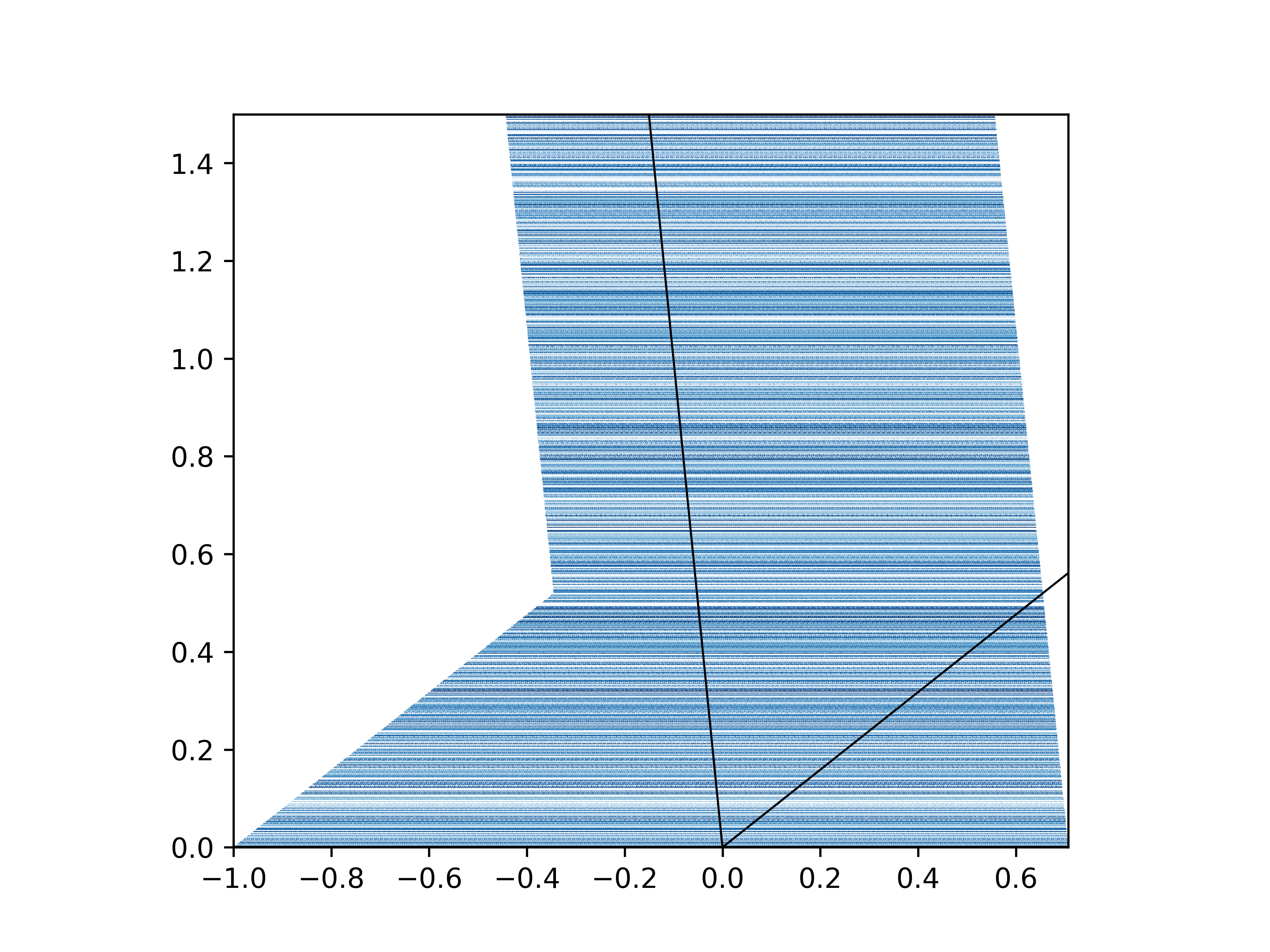}}
    \caption{A plot of the first 2000 iterates (after removing 1000 iterates to omit transients) of 1500 points, chosen uniformly within the box $[-1,\lambda] \times [0,1.5]$, under a TCE of the form \eqref{eq:parameters1} with parameters $\alpha = (\frac{\pi}{2}-0.9,1,\frac{\pi}{2}-0.1)$ and $\lambda = \frac{\sqrt{2}}{2}$.  Observe that the attractor here is stratified since the map consists only of horizontal translations.}
    \label{fig:TCEnorot}
\end{figure}

We show that the maximal invariant set in $\ol{\bbH}$ is the set depicted in Figure~\ref{fig:TCEnorot} and the dynamics of this consists of an interpolation between two Interval exchanges, one on each end.

Since $\tau_0 = \tau_1 = \tau_2 = 0$, the map $F$ is a piecewise translation which preserves horizontal lines $\R + iy = \{t + iy : t \in \R\}$, for all $y \geq 0$.  In particular, the action of $F$ on $\R+iy$ is conjugate to an interval translation map $h_y: \R \rightarrow \R$,
\begin{equation*}
h_y(x) =    \begin{cases}
            x - 1       &\text{if $x > \frac{y}{\tan\alpha_0}$}, \\
            x - \eta    &\text{if $-\frac{y}{\tan\alpha_2} < x < \frac{y}{\tan\alpha_0}$}, \\
            x + \lambda &\text{if $x < -\frac{y}{\tan\alpha_2}$},
            \end{cases}
\end{equation*}
via the conjugation
\begin{equation}
\label{eq:Fhconjugation}
F|_{\R+iy}(z) = h_y(z-iy)+iy.
\end{equation}

\begin{theorem}
\label{InvSetThm}
Let $F$ be a TCE \eqref{eq:F} with parameters \eqref{eq:parameters1} and let $M$ denote the set
\begin{equation}
\label{eq:M}
M = ((P_0 - 1) \cup (P_1 - \eta)) \cap (P_2 + \lambda).
\end{equation}
Then $M$ is the maximal invariant set and global attractor for $F$. For all $y\geq 0$, let
\begin{equation} 
M_y:=M\cap (\R +iy)
\label{eq:foliation}
\end{equation}
and note that this is an invariant set for $F$. The dynamics on $M$ consist of an IET on each $M_y$. Moreover, there is a $y^*>0$ such that there are 3-IETs for each $M_y$ with $0<y<y^*$ that interpolate between one 2-IET on $M_0$ and another 2-IET on each $M_y$ with $y\geq y^*$.
\end{theorem}

\begin{proof}
We will first prove that for all $z \in \ol{\mathbb{H}}$, there is some $N \in \N$ for which $F^N(z) \in M$.  Then it will suffice to show that $M$ is $F$-invariant.

Firstly, observe that we can segment $P_0$ into ribbons (see figure \ref{fig:slices})
\begin{equation*}
P_0 = \bigcup_{n=0}^\infty (P_0 + n) \setminus (P_0 + n + 1),
\end{equation*}
so that for $n \geq 0$
\begin{equation*}
F((P_0 + n) \setminus (P_0 + n + 1)) = (P_0 + n-1) \setminus (P_0 + n).
\end{equation*}

\begin{figure}
    \centering
    \includegraphics[width=0.8\linewidth]{{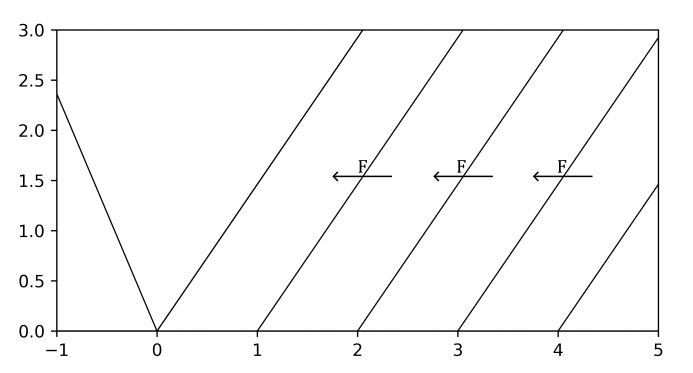}}
    \caption{An illustration of a construction used for the proof of Theorem \ref{eq:foliation}, which is a partition of $P_0$ into ribbons on which $F$ acts by shifting every ribbon to the one on its left.}
    \label{fig:slices}
\end{figure}

Therefore, for all $z \in P_0$, there is some $n \in \N$ for which $z \in (P_0 + n) \setminus (P_0 + n + 1)$, from which it follows that $F^{n+1}(z) \in (P_0 - 1) \setminus P_0$, i.e.
\begin{equation*}
F^{n+1}(z) \in (P_0 - 1) \cap (P_1 \cup P_2).
\end{equation*}
Thus, either $F^{n+1}(z) \in (P_0 - 1) \cap P_2 \subset M$, in which case we are done, or $F^{n+1}(z) \in P_1$.

Now suppose that $z \in P_1$.  Once again, we can partition $P_1$ into a union
\begin{equation*}
P_1 = \bigcup_{n=0}^\infty (P_1 + n\eta) \setminus (P_1 + (n+1)\eta).
\end{equation*}
with the property that
\begin{equation*}
F((P_1 + n\eta) \setminus (P_1 + (n+1)\eta)) = (P_1 + (n-1)\eta) \setminus (P_1 + n\eta).
\end{equation*}
It is clear that given $z \in P_1$, there is some $n \in \N$ for which $z \in (P_1 + n\eta) \setminus (P_1 + (n+1)\eta)$ so that 
\begin{equation*}
F^{n+1}(z) \in (P_1 - \eta) \setminus P_1 = (P_1 - \eta) \cap (P_0 \cup P_2).
\end{equation*}
It is clear that $P_1 - \eta \subset P_1 \cup P_2$, so in fact
\begin{equation*}
F^{n+1}(z) \in (P_1 - \eta) \cap P_2 \subset M.
\end{equation*}

Now suppose that $z \in P_2$.  Similarly as before we can segment $P_2$ into ribbons of the form
\begin{equation*}
P_2 = \bigcup_{n=1}^\infty (P_2 - n\lambda) \setminus (P_2 - (n+1)\lambda),
\end{equation*}
so that
\begin{equation*}
F((P_2 - n\lambda) \setminus (P_2 - (n+1)\lambda)) = (P_2 - (n-1)\lambda) \setminus (P_2 - n\lambda).
\end{equation*}
It is clear to see that for each $z \in P_2$, there is some $n \in \N$ such that $z \in (P_2 - n\lambda) \setminus (P_2 - (n+1)\lambda)$.  Hence, 
\begin{equation*}
F^{n+1}(z) \in (P_2 + \lambda) \setminus P_2 = (P_2 + \lambda) \cap (P_0 \cup P_1).
\end{equation*}
Now if $F^{n+1}(z) \in (P_2 + \lambda) \cap P_0$ then we are done, since $P_0 \subset (P_0 - 1)$.  Now it remains to show that $P_1 \cap (P_2 + \lambda) \subset M$.  But note that $P_1 \cap (P_2 - \eta) = \emptyset$, since $(P_2 - \eta) \subset P_2$ and $P_1 \cap P_2 = \emptyset$.  Therefore,
\begin{equation*}
P_1 \subset (P_0 - \eta) \cup (P_1 - \eta) \subset (P_0 - 1) \cap (P_1 - \eta),
\end{equation*}
from which it follows that $P_1 \cap (P_2 + \lambda) \subset M$.

We will now show that $M$ is $F$-invariant.
Recall \eqref{eq:foliation} and let $M_y' = M_y - iy$.  Examining the components of the intersection \eqref{eq:M} defining $M$, we derive the following inequalities.  Firstly, from the $P_2 + \lambda$ component we see that for all $z \in M$, 
\begin{equation}
\label{eq:P2inequality}
\Re(z) < \lambda - \frac{\Im(z)}{\tan\alpha_2},
\end{equation}
Secondly, from the component $(P_0 - 1) \cup (P_1 - \eta)$, we find that for all $z \in M$,
\begin{equation}
\label{eq:P0P1inequalities}
\Re(z) > \frac{\Im(z)}{\tan\alpha_0} - 1, \text{~or } -\eta - \frac{\Im(z)}{\tan\alpha_2} < \Re(z) < -\eta + \frac{\Im(z)}{\tan\alpha_0}.
\end{equation}
Combining \eqref{eq:P2inequality} and \eqref{eq:P0P1inequalities}, we see that for all $z = x+iy \in M$,
\begin{equation*}
\min\left\{-1+\frac{y}{\tan\alpha_0},-\eta-\frac{y}{\tan\alpha_2}\right\} < x < \lambda-\frac{y}{\tan\alpha_2}.
\end{equation*}
Now to achieve a concrete definition of the map $h_y|_{M_y'}$, we now wish to find similar inequalities for $M \cap P_j$ for $j=0,1,2$.  Trivially, we see that $x+iy \in M \cap P_0$ if and only if
\begin{equation*}
\frac{y}{\tan\alpha_0} < x < \lambda - \frac{y}{\tan\alpha_2},
\end{equation*}
and $x+iy \in M \cap P_2$ if and only if
\begin{equation*}
\min\left\{-1+\frac{y}{\tan\alpha_0},-\eta-\frac{y}{\tan\alpha_2}\right\} < x < -\frac{y}{\tan\alpha_2}.
\end{equation*}
Finally, $x+iy \in M \cap P_1$ if and only if
\begin{equation*}
\max\left\{\min\left\{-1+\frac{y}{\tan\alpha_0},-\eta-\frac{y}{\tan\alpha_2}\right\},-\frac{y}{\tan\alpha_2}\right\} < x < \min\left\{\frac{y}{\tan\alpha_0},\lambda - \frac{y}{\tan\alpha_0}\right\}.
\end{equation*}
This can be simplified if we observe that
\begin{equation*}
\min\left\{-1+\frac{y}{\tan\alpha_0},-\eta-\frac{y}{\tan\alpha_2}\right\} \leq -\eta-\frac{y}{\tan\alpha_2} < -\frac{y}{\tan\alpha_2}.
\end{equation*}
Thus, $x+iy \in M \cap P_1$ if and only if
\begin{equation*}
-\frac{y}{\tan\alpha_2} < x < \min\left\{\frac{y}{\tan\alpha_0},\lambda - \frac{y}{\tan\alpha_0}\right\}.
\end{equation*}
We thus see that the map $h_y|_{M_y'}$ is defined by
\begin{equation}
\label{eq:IETslice}
h_y|_{M_y'}(x) =    \begin{cases}
            x - 1       &\text{if $\frac{y}{\tan\alpha_0} < x < \lambda - \frac{y}{\tan\alpha_2}$}, \\
            x - \eta    &\text{if $-\frac{y}{\tan\alpha_2} < x < \min\left\{\frac{y}{\tan\alpha_0},\lambda - \frac{y}{\tan\alpha_2}\right\}$}, \\
            x + \lambda &\text{if $\min\left\{-1+\frac{y}{\tan\alpha_0},-\eta-\frac{y}{\tan\alpha_2}\right\} < x < -\frac{y}{\tan\alpha_2}$},
                \end{cases}
\end{equation}
Let $y^* > 0$ such that $\frac{y^*}{\tan\alpha_0} = \lambda - \frac{y^*}{\tan\alpha_2}$.  Then through classic trigonometric identities, we can deduce that
\begin{equation}
y^* = \lambda\frac{\sin\alpha_0\sin\alpha_2}{\sin(\alpha_0+\alpha_2)}.
\label{eq:ystar}
\end{equation}
Now we can examine the map $h_y|_{M_y'}$ in the cases where $y=0$, $0<y<y^*$ and $y \geq y^*$.  Firstly, in the case of $y=0$, the definition \eqref{eq:IETslice} simplifies to
\begin{equation*}
h_0|_{M_0'}(x) = \begin{cases}
            x - 1       &\text{if $0 < x < \lambda$}, \\
            x + \lambda &\text{if $-1 < x < 0$},
                \end{cases}
\end{equation*}
from which it is trivial to see this is a IET of two intervals.  Secondly, the case $y \geq y^*$ corresponds to the inequality $\frac{y}{\tan\alpha_0} > \lambda - \frac{y}{\tan\alpha_2}$ and equivalently $-1+\frac{y}{\tan\alpha_0} > -\eta - \frac{y^*}{\tan\alpha_2}$. By combining these inequalities with \eqref{eq:IETslice}, we see that
\begin{equation*}
h_y|_{M_y'}(x) =    \begin{cases}
            x - \eta    &\text{if $-\frac{y}{\tan\alpha_2} < x < \lambda - \frac{y}{\tan\alpha_2}$}, \\
            x + \lambda &\text{if $-\eta-\frac{y}{\tan\alpha_2} < x < -\frac{y}{\tan\alpha_2}$},
                \end{cases}
\end{equation*}
which is trivially also an IET of two intervals.
Finally, we also see that for all $0 < y < y^*$,
\begin{equation*}
h_y|_{M_y'}(x) =    \begin{cases}
            x - 1       &\text{if $\frac{y}{\tan\alpha_0} < x < \lambda - \frac{y}{\tan\alpha_2}$}, \\
            x - \eta    &\text{if $-\frac{y}{\tan\alpha_2} < x < \frac{y}{\tan\alpha_0}$}, \\
            x + \lambda &\text{if $-1+\frac{y}{\tan\alpha_0} < x < -\frac{y}{\tan\alpha_2}$}.
                \end{cases}
\end{equation*}
We see that
\begin{equation*}
\begin{aligned}
\left( \frac{y}{\tan\alpha_0}, \lambda - \frac{y}{\tan\alpha_2} \right) - 1 = \left( -1 + \frac{y}{\tan\alpha_0}, -\eta - \frac{y}{\tan\alpha_2} \right), \\
\left( -\frac{y}{\tan\alpha_2}, \frac{y}{\tan\alpha_0} \right) - \eta = \left( -\eta - \frac{y}{\tan\alpha_2}, -\eta + \frac{y}{\tan\alpha_0} \right), \\
\left( -1 + \frac{y}{\tan\alpha_0}, -\frac{y}{\tan\alpha_2} \right) + \lambda = \left( -\eta + \frac{y}{\tan\alpha_0}, \lambda-\frac{y}{\tan\alpha_2} \right).
\end{aligned}
\end{equation*}
It is clear from this that for $0 < y < y^*$, $h_y|_{M_y'}$ is length-preserving, hence it is invertible, and therefore is an IET of three intervals.
In all three cases, the set $M_y'$ is $h_y$-invariant.  Therefore, since \eqref{eq:Fhconjugation} and $M = \bigcup_{y>0}M_y$, we conclude that $M$ is $F$-invariant, which completes the proof.
\end{proof}

The conjugation \eqref{eq:Fhconjugation} implies that $M_y'$ is the maximal invariant set and attractor for $h_y$.

To re-frame the one-dimensional slices $M_y$ into a two-dimensional perspective, we can observe that the set
\begin{equation*}
T = \left\{x+iy \in \C : -1+\frac{y}{\tan\alpha_0} < x < \lambda - \frac{y}{\tan\alpha_2}, 0 < y < y^*\right\},
\end{equation*}
which is also clearly equal to the union $\bigcup_{0 < y < y^*} M_y$, is an invertible PWI of three atoms $T \cap P_0$, $T \cap P_1$ and $T \cap P_2$, for which the bottom and top edges ($M_0$ and $M_{y^*}$ respectively) are IETs of two intervals.  In fact, the map $h_{y^*}|_{M_{y^*}'}$ is conjugate (via translation) to the first Rauzy induction of $h_0|_{M_0'}$, that is the first return map of $h_0$ to $[-1,0]$. See figure \ref{fig:TrapeziumMap} for an illustration of this mapping.  One could say that the map $(x,y) \mapsto h_y|_{M_y'}(x)$ is a form of ``homotopy'' between the IETs $h_0|_{M_0'}$ and $h_{y^*}|_{M_{y^*}'}$, although this function is only piecewise continuous and thus does not achieve the criteria to be a homotopy in the usual sense.

\begin{figure}
\centering
\includegraphics[width=0.7\linewidth]{{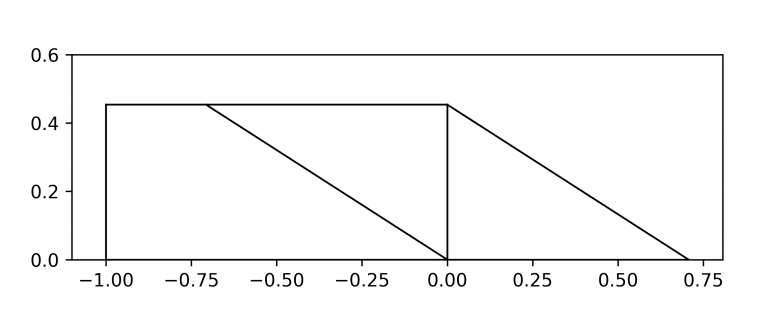}}

(a)

~

\includegraphics[width=0.7\linewidth]{{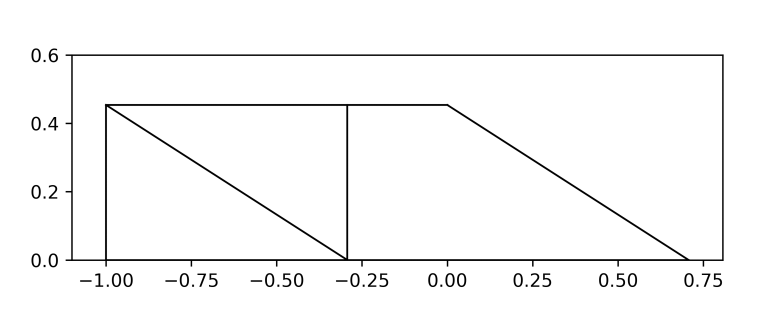}}

(b)

\caption{(a): The partition of the trapezium $T$, with cone angles $\alpha = (\pi/2,1,\pi/2-1)$ and translation parameter $\lambda =\sqrt{2}/2$, which gives rise to a piecewise translation.  Note that the discontinuity lines split the top and bottom boundaries into two subintervals. (b) The partition from (a) now transformed under the PWI.  Note that the subintervals of the top and bottom boundaries in each atom have exchanged like an IET.}
\label{fig:TrapeziumMap}
\end{figure}

On the other hand, the set 
\begin{equation*}
T' = \left\{x+iy \in \C : -\eta - \frac{y}{\tan\alpha_2} < x < \lambda - \frac{y}{\tan\alpha_2}, y \geq y^*\right\},
\end{equation*}
is a strip on which for every $y \geq y^*$, the map $h_y|_{M_y'}$ is conjugate via translations to $h_{y^*}|_{M_{y^*}'}$, i.e. $F$ acts on $T'$ by exchanging the strips
\begin{equation*}
\left\{x+iy \in \C : -\eta-\frac{y}{\tan\alpha_2} < x < -\frac{y}{\tan\alpha_2}, y \geq y^*\right\},
\end{equation*}
and
\begin{equation*}
\left\{x+iy \in \C : -\frac{y}{\tan\alpha_2} < x < \lambda - \frac{y}{\tan\alpha_2}, y \geq y^*\right\}.
\end{equation*}

\section{Some rotational TCEs with families of polygonal invariant curves}

For any value of  $0 < \phi < \frac{\pi}{2}$ and $0<\lambda<1$ we consider the TCE (\ref{eq:F}) for $d=2$ with
\begin{equation}
\label{parameters}
\begin{aligned}
\alpha 	&= \left( \frac{\pi}{2} - \phi, \phi, \phi, \frac{\pi}{2} - \phi \right), \\
\sigma  & = (0)(1~ 2)(3) 
\end{aligned}
\end{equation}
Define $\eta:=1-\lambda$ and 
\begin{equation}
N(\lambda,\phi):= \left\lfloor \frac{\lambda}{2\eta(1+\cos \phi)}\right\rfloor.
\label{eq:Ndef}
\end{equation}
and note that $N\in\N$ can be arbitrarily large by choosing $\lambda$ close to 1 (i.e. $\eta$ close to zero).

We claim that for each $n=1,\ldots,N$ there will be a linear embedding of an IET into the TCE with this data.

To this end we define
\begin{equation*}
l^{(n)}	= \lambda - 2n(1 + \cos\phi)\eta,
\end{equation*}
and note that $0 \leq l^{(n)} < l^{(n-1)}$ for all $1 \leq n \leq N$.

\begin{remark}
The conditions that $l^{(n)} \geq 0$ and $\lambda < 1$ are equivalent to the inequality
\begin{equation}
\label{initialineq}
1 - \frac{1}{2n(1 + \cos\phi) + 1} \leq \lambda < 1.
\end{equation}
This confirms that as $\lambda \nearrow 1$, $N(\lambda,\phi) \geq n$ for arbitrarily large $n \in \N$, and in fact for any $n \in \N$ $N(\lambda,\phi) = n$ if and only if
\begin{equation*}
1 - \frac{1}{2n(1 + \cos\phi) + 1} \leq \lambda < 1 - \frac{1}{2(n+1)(1 + \cos\phi) + 1}.
\end{equation*}
\end{remark}

Our aim is to prove the existence of a family of linear embeddings of IETs into TCEs of the kind \eqref{parameters}, and to do this we will first note that piecewise linear curves can be specified by the vertices that form the endpoints of each linear segment of the curve.

We will now define a finite sequence of points $(a_k^{(n)})_{k = 0}^{4n+3}$ (see figure \ref{morecurves}) by
\begin{equation}
\label{points}
\begin{aligned}
a_0^{(n)} 		&= n\eta(1 + \cos\phi) - 1 + in\eta \frac{1 + \cos\phi}{\tan\phi}, \\
a_1^{(n)} 		&= a_0^{(n)} + l^{(n)} = - \eta - n\eta(1 + \cos\phi) + in\eta \frac{1 + \cos\phi}{\tan\phi},\\
a_{k+1}^{(n)} 	&= 	\begin{cases}
				a_k^{(n)} + \eta 		& \text{ if $k$ is odd,} \\
				a_k^{(n)} + \eta\exp(i\phi) 	& \text{ if $k$ is even, } 1 < k \leq 2n+1, \\
				a_k^{(n)} + \eta\exp(-i\phi)	& \text{ if $k$ is even, } 2n+1 < k \leq 4n+2,
				\end{cases} \\
a_{4n+3}^{(n)}	&= a_{4n+2}^{(n)} + l^{(n)}.
\end{aligned}
\end{equation}
Let $I^{(n)} = \left[ 0, 2l^{(n)} + (4n+1)\eta \right)$ and define the points $t_k^{(n)}$, for $k \in \{0,\ldots,4n+2\}$, by
\begin{equation*}
t_k^{(n)} = \begin{cases}
            0,                      &\text{ if $k = 0$,} \\
            l^{(n)} + (k-1)\eta,        &\text{ if $k \in \{1,\ldots,4n+2\}$,} \\
            2l^{(n)} + (4n+1)\eta   &\text{ if $k = 4n+3$.}
            \end{cases}
\end{equation*} 
The sequence $(a_k^{(n)})$ defines the vertices for a piecewise linear curve $\gamma^{(n)} : I^{(n)} \rightarrow \ol{\mathbb{H}}$, that is
\begin{equation*}
\gamma^{(n)}(t) = \frac{t - t_k^{(n)}}{t_{k+1}^{(n)} - t_k^{(n)}}(a_{k+1}^{(n)} - a_k^{(n)}) + a_k^{(n)} ~\text{ for $t_k^{(n)} \leq t \leq t_{k+1}^{(n)}$,}
\end{equation*}
The curve $\gamma^{(n)}$ is the linear interpolation of the points $a_0^{(n)}$, $a_1^{(n)}$, ...,$a_{4n+3}^{(n)}$. We shall call $\gamma^{(n)}$ an \textit{$n$-stepped cap}.  

Given points $x_0, x_1, ..., x_{M-1} \in \C$, we will denote the open and bounded $M$-sided polygon with vertices $x_0$,...,$x_{M-1}$ and edges $(x_k,x_{k+1})$ and $(x_{M-1},x_0)$ by 
\begin{equation}
\label{polygon}
[x_0,\ldots,x_{M-1}].    
\end{equation}
When $N(\lambda, \phi) \geq n \geq 1$, let $Y^{(n)}$ denote the polygon
\begin{equation}
\label{Yregion}
Y^{(n)} = [-1,a_0^{(n)}, a_1^{(n)}, \ldots, a_{4n+3}^{(n)},\lambda].
\end{equation}
We call $Y^{(n)}$ an \textit{$n$-stepped pyramid}.  See figures \ref{morecurves} and \ref{mappedcurves}.

\begin{figure}
\centering
\includegraphics[width=0.8\linewidth]{{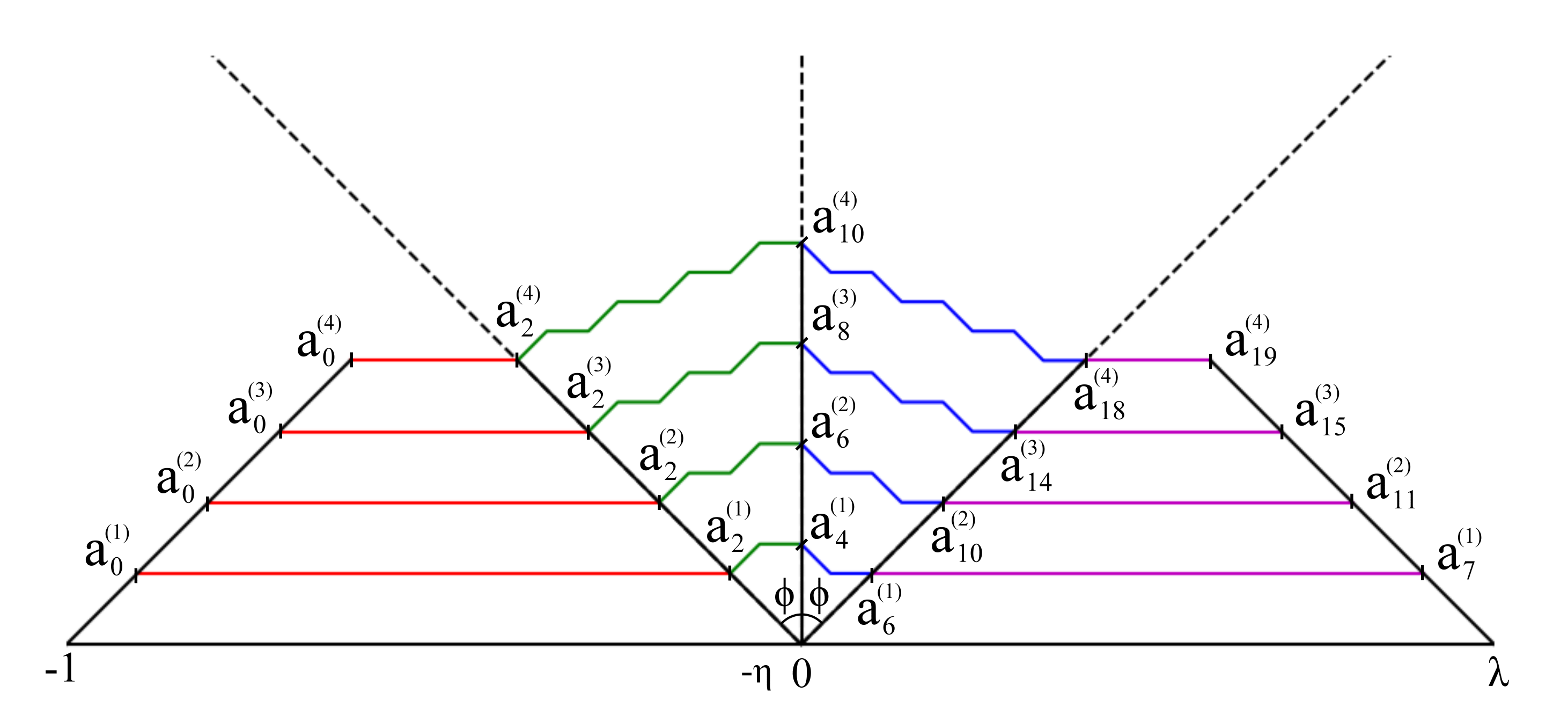}}
\caption{An illustration of $n$-stepped caps for $n = 1,2,3,4$ ascending from the bottom curve, with the same parameters as in figure \ref{094345}, which each form the ``cap'' of the boundary for their respective $n$-stepped pyramid.  Each vertex of the curves represents one of the points $a_k^{(n)}$, though only some points have been labelled.  The curves $\gamma^{(n)}$ have been coloured into the segments $\gamma^{(n)}(I_k^{(n)})$ from \eqref{curve partition}, i.e. $\gamma^{(n)}(I)$ is partitioned according to which cone the segment belongs to.}
\label{morecurves}
\end{figure}

\begin{figure}
\centering
\includegraphics[width=0.8\linewidth]{{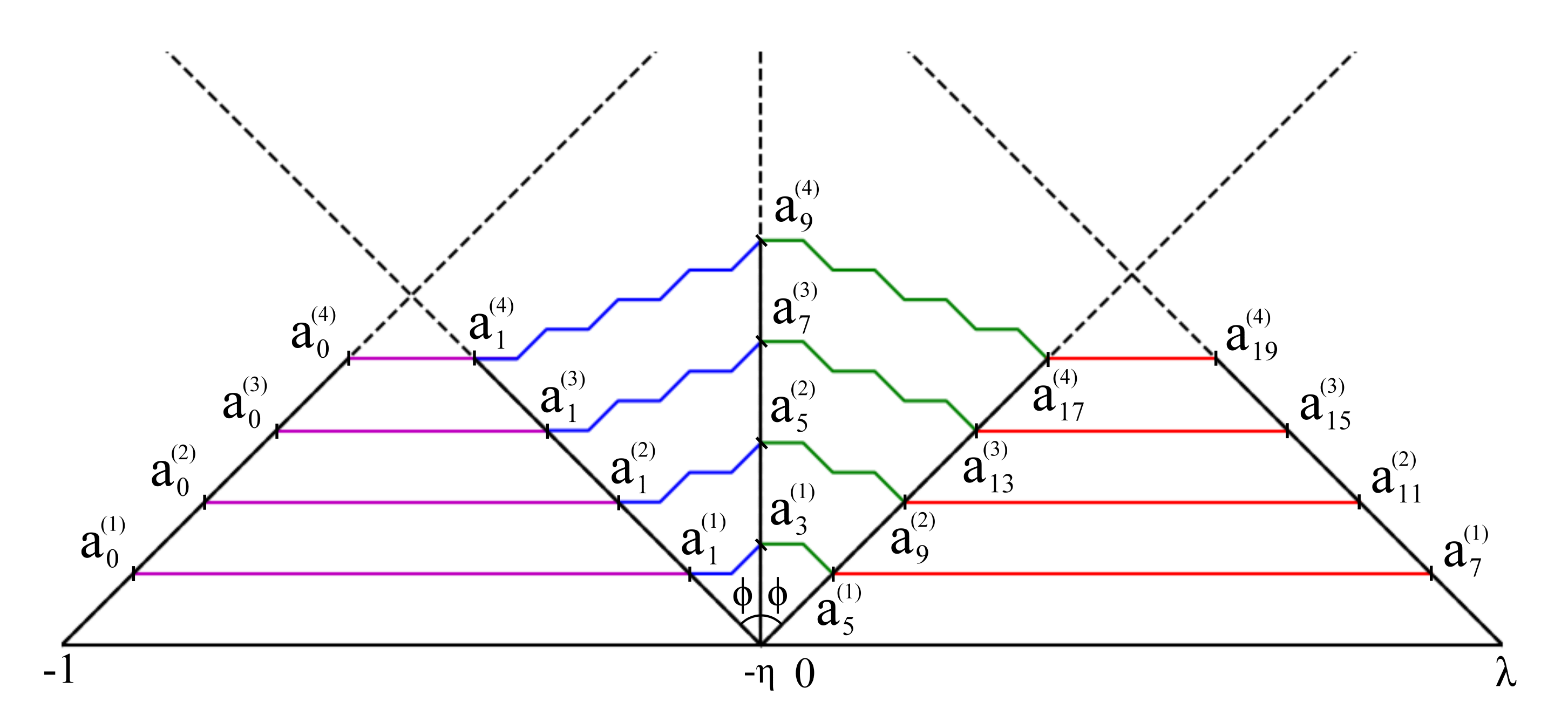}}
\caption{The $n$-stepped caps as in figure \ref{morecurves}, but mapped under the TCE $F$ with the same parameters.  Observe that the curves seem to be preserved, although the individually coloured pieces are rearranged within each curve into the partition $\gamma^{(n)}(I_k^{(n)}{}')$ as in \eqref{gammaprime}, which is akin to an interval exchange transformation.  Additionally note that the $n$-stepped pyramids also appear to be preserved under $F$.}
\label{mappedcurves}
\end{figure}

\begin{theorem}
\label{invariantcurves}
Consider the TCE $(\calp,F)$ defined by (\ref{eq:F}) with parameters (\ref{parameters}) and define $N=N(\lambda,\phi)$ as in (\ref{eq:Ndef}). Then for each $0\leq n\leq N$ there is an $n$-stepped cap $\gamma^{(n)}$ which is $F$-invariant and is a linear embedding of a 4-IET into $(\calp,F)$.  Moreover, when $N \geq 1$, for each $1 \leq n \leq N$ there is an $n$-stepped pyramid $Y^{(n)}$ that is $F$-invariant.
\end{theorem}

\begin{proof}
Let $0 \leq n \leq N$ be an integer.  Observe that
\begin{equation}
\label{a2}
\begin{aligned}
a_2^{(n)}	&= a_1^{(n)} + \eta \\
		&= - n\eta(1 + \cos\phi) + in\eta \frac{1 + \cos\phi}{\tan\phi} \\
		&= n\eta \frac{1 + \cos\phi}{\sin\phi} (-\sin\phi + i\cos\phi) \\
		&= n\eta \frac{1 + \cos\phi}{\sin\phi} \exp\left(i \left(\phi + \frac{\pi}{2} \right)\right) \in \ol{P_2} \cap \ol{P_3}.
\end{aligned}
\end{equation}
Secondly, noting that
\begin{equation*}
\sin\phi + \frac{1 + \cos\phi}{\tan\phi} = \frac{1}{\sin\phi}( (\sin\phi)^2 + \cos\phi + (\cos\phi)^2 ) = \frac{1 + \cos\phi}{\sin\phi},
\end{equation*}
we have
\begin{equation}
\label{a2n}
\begin{aligned}
a_{2n+2}^{(n)}	&= a_2^{(n)} + \sum_{j=1}^n \eta + \sum_{j=1}^n \eta \exp(i\phi) \\
		&= - n\eta(1 + \cos\phi) + in\eta \frac{1 + \cos\phi}{\tan\phi} + n\eta + n\eta \exp(i\phi), \\
		&= in\eta \left( \sin\phi + \frac{1 + \cos\phi}{\tan\phi} \right) \\
		&= in\eta \frac{1 + \cos\phi}{\sin\phi} \in \ol{P_1} \cap \ol{P_2}.
\end{aligned}
\end{equation}
Thirdly, we have
\begin{equation}
\label{a4n}
\begin{aligned}
a_{4n+2}^{(n)}	&= a_{2n+2}^{(n)} + \sum_{j=1}^n \eta + \sum_{j=1}^n \eta \exp(-i\phi) \\
		&= in\eta \left( \sin\phi + \frac{1 + \cos\phi}{\tan\phi} \right) + n\eta + n\eta \exp(-i\phi), \\
		&= in\eta \sin\phi + in\eta\frac{1 + \cos\phi}{\tan\phi} + n\eta + n\eta \cos\phi - in\eta \sin\phi \\
		&= n\eta \frac{1 + \cos\phi}{\sin\phi} \left( \sin\phi + i\cos\phi \right) \\
		&= n\eta \frac{1 + \cos\phi}{\sin\phi} \exp\left(i\left( \frac{\pi}{2} - \phi \right)\right) \in \ol{P_0} \cap \ol{P_1}.
\end{aligned}
\end{equation}
We will partition the interval $I^{(n)}$ into four subintervals
\begin{equation}
\label{curve partition}
\begin{aligned}
I_0^{(n)}   &\supset (t_0^{(n)},t_2^{(n)}), \\
I_1^{(n)}   &\supset (t_2^{(n)},t_{2n+2}^{(n)}), \\
I_2^{(n)}   &\supset (t_{2n+2}^{(n)},t_{4n+2}^{(n)}), \\
I_3^{(n)}   &\supset (t_{4n+2}^{(n)},t_{4n+3}^{(n)}), \\
\end{aligned}
\end{equation}
so that $\gamma(I_j^{(n)}) \subset P_{3-j}$ for each $j$ by the above discussion, in particular \eqref{a2}, \eqref{a2n} and \eqref{a4n}.  

\begin{figure}
\centering
\includegraphics[width=\linewidth]{{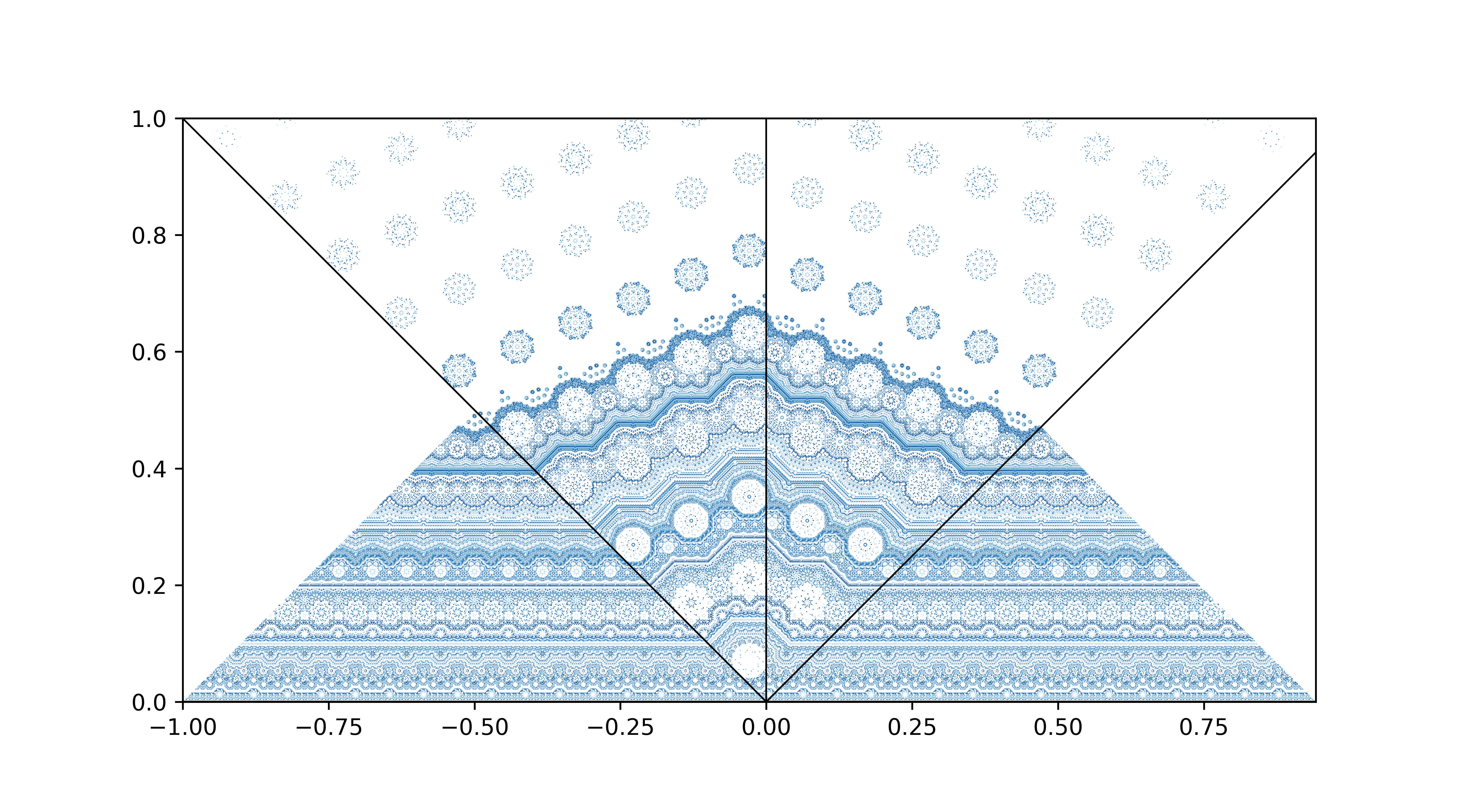}}
\caption{A plot of the first 1500 iterates (after removing 1000 iterates to omit transients) of 1000 points, chosen uniformly within the box $[-1,\lambda] \times [0,1]$, under the TCE with parameters $\phi = \pi/4$ and $\lambda = \frac{8 + 4\sqrt{2}}{9 + 4\sqrt{2}} + 0.01$.  Note the appearance of more curves with an increasing number of steps.}
\label{094345}
\end{figure}

Define another partition of $I^{(n)}$ by
\begin{equation}
\label{gammaprime}
\begin{aligned}
I_0^{(n)\prime}   &\supset (t_0^{(n)},t_1^{(n)}), \\
I_1^{(n)\prime}   &\supset (t_1^{(n)},t_{2n+1}^{(n)}), \\
I_2^{(n)\prime}   &\supset (t_{2n+1}^{(n)},t_{4n+1}^{(n)}), \\
I_3^{(n)\prime}   &\supset (t_{4n+1}^{(n)},t_{4n+3}^{(n)}). \\
\end{aligned}
\end{equation}
Then to prove the $F$-invariance of $\gamma^{(n)}(I^{(n)})$, it suffices to show that
\begin{equation}
\label{proofcondition}
F(\gamma^{(n)}(I_j^{(n)})) = \gamma^{(n)}(I_{3-j}^{(n)\prime}),
\end{equation}
for $j=0,1,2,3$.

Since the curve $\gamma^{(n)}$ is determined by the points $a_k^{(n)}$, we need only prove that \eqref{proofcondition} holds on the appropriate $a_k^{(n)}$.

Observe that by \eqref{a4n},
\begin{equation*}
\begin{aligned}
a_0^{(n)} 	&= n\eta(1 + \cos\phi) - 1 + in\eta \frac{1 + \cos\phi}{\tan\phi} \\
		&= n\eta \frac{1 + \cos\phi}{\sin\phi} \left( \sin\phi + i\cos\phi \right) - 1 \\
		&= a_{4n+2}^{(n)} - 1.
\end{aligned}
\end{equation*}
Hence,
\begin{equation*}
a_1^{(n)} = a_0^{(n)} + l^{(n)} = a_{4n+2}^{(n)} - 1 + l^{(n)} = a_{4n+3}^{(n)} - 1,
\end{equation*}
so by \eqref{curve partition} we have
\begin{equation*}
F(\gamma^{(n)}(I_3^{(n)})) = \gamma^{(n)}(I_3^{(n)}) - 1 = \gamma^{(n)}(I_0^{(n)\prime}).
\end{equation*}
Similarly, we have
\begin{equation*}
a_{4n+1}^{(n)} = a_{4n+2}^{(n)} - \eta = a_{4n+2} - 1 + \lambda = a_0^{(n)} + \lambda,
\end{equation*}
from which we deduce
\begin{equation*}
a_{4n+3}^{(n)} = a_{4n+1}^{(n)} + \eta + l^{(n)} = a_0^{(n)} + \lambda + \eta + l^{(n)} = a_2^{(n)} + \lambda,
\end{equation*}
and thus by \eqref{curve partition}, we get
\begin{equation*}
F(\gamma^{(n)}(I_0^{(n)})) = \gamma^{(n)}(I_0^{(n)}) + \lambda = \gamma^{(n)}(I_3^{(n)\prime}).
\end{equation*}
We know that
\begin{equation*}
F(\gamma^{(n)}(I_1^{(n)})) = \exp(-i\phi)\gamma^{(n)}(I_1^{(n)}) - \eta, 
\end{equation*}
since $\gamma^{(n)}(I_1^{(n)}) \subset P_2$.  Recall from \eqref{points} that for $1 < k \leq 2n+1$
\begin{equation*}
a_{k+1}^{(n)} = 	\begin{cases}
			a_k^{(n)} + \eta 			&\text{ if $k$ is odd}, \\
			a_k^{(n)} + \eta \exp(i\phi)	&\text{ if $k$ is even}.
			\end{cases}
\end{equation*}
From this, it follows from distributing and rearranging terms that
\begin{equation}
\label{induction1step}
\begin{aligned}
\exp(-i\phi) a_{k+1}^{(n)} - \eta	&= 	\begin{cases}
						\exp(-i\phi)( a_k^{(n)} + \eta ) - \eta 			&\text{ if $k$ is odd}, \\
						\exp(-i\phi)( a_k^{(n)} + \eta \exp(i\phi) ) - \eta	&\text{ if $k$ is even},
						\end{cases} \\
					&= 	\begin{cases}
						(\exp(-i\phi)a_k^{(n)} - \eta) + \eta \exp(-i\phi)		&\text{ if $k$ is odd}, \\
						(\exp(-i\phi)a_k^{(n)} - \eta) + \eta			&\text{ if $k$ is even}.
						\end{cases}
\end{aligned}
\end{equation}
In addition, we can see from comparison of \eqref{a2} and \eqref{a2n}
\begin{equation}
\label{comparison}
a_{2n+2}^{(n)} = \exp(-i\phi)a_2^{(n)}.
\end{equation}
By using \eqref{points} in the case $k = 2n+1$, we see that
\begin{equation}
\label{induction1base}
a_{2n+1}^{(n)} = a_{2n+2}^{(n)} - \eta = \exp(-i\phi)a_2^{(n)} - \eta.
\end{equation}
Using \eqref{induction1base} and \eqref{induction1step}, we can use an inductive argument to show that for $1 \leq k \leq 2n+1$,
\begin{equation}
\label{shiftinga1}
\exp(-i\phi)a_{k+1}^{(n)} - \eta = a_{2n + k}^{(n)}.
\end{equation}
Recalling the definition of $I_2^{(n)\prime}$ in \eqref{gammaprime}, we know that $\gamma^{(n)}(I_2^{(n)\prime})$ is the linear interpolation of the points, in order, $a_{2n+1}^{(n)}$, ...,$a_{4n+1}^{(n)}$.  Thus, it follows from this and \eqref{shiftinga1} that
\begin{equation*}
F(\gamma^{(n)}(I_1^{(n)})) = \exp(-i\phi)\gamma^{(n)}(I_1^{(n)}) - \eta = \gamma^{(n)}(I_2^{(n)\prime}).
\end{equation*}
In a similar fashion, recall from \eqref{curve partition} that $\gamma^{(n)}(I_2^{(n)}) \subset P_1$, so
\begin{equation*}
F(\gamma^{(n)}(I_2^{(n)})) = \exp(i\phi)\gamma^{(n)}(I_2^{(n)}) - \eta.
\end{equation*}
Recall from \eqref{points} that for $2n+1 < k \leq 4n+1$ that
\begin{equation*}
a_{k+1}^{(n)} = 	\begin{cases}
			a_k^{(n)} + \eta 			&\text{ if $k$ is odd}, \\
			a_k^{(n)} + \eta \exp(-i\phi)	&\text{ if $k$ is even}.
			\end{cases}
\end{equation*}
We thus have
\begin{equation}
\label{induction2step}
\begin{aligned}
\exp(i\phi)a_{k+1}^{(n)} - \eta	&= 	\begin{cases}
						\exp(i\phi)( a_k^{(n)} + \eta ) - \eta 			&\text{ if $k$ is odd}, \\
						\exp(i\phi)( a_k^{(n)} + \eta \exp(-i\phi) ) - \eta	&\text{ if $k$ is even},
						\end{cases} \\
					&= 	\begin{cases}
						(\exp(i\phi)a_k^{(n)} - \eta) + \eta \exp(i\phi)		&\text{ if $k$ is odd}, \\
						(\exp(i\phi)a_k^{(n)} - \eta) + \eta				&\text{ if $k$ is even}.
						\end{cases}
\end{aligned}
\end{equation}
Recalling \eqref{comparison}, and using \eqref{points} in the case $k = 1$, we see that
\begin{equation}
\label{induction2base}
a_1^{(n)} = a_2^{(n)} - \eta = \exp(i\phi)a_{2n+2}^{(n)} - \eta.
\end{equation}
Combining \eqref{induction2step} and \eqref{induction2base}, we can use another inductive argument to show that for $2n+1 \leq k \leq 4n+1$,
\begin{equation}
\label{shiftinga2}
\exp(i\phi)a_{k+1}^{(n)} - \eta = a_{k - 2n}^{(n)}.
\end{equation}
It is clear from the definition of $I_1^{(n)\prime}$ in \eqref{gammaprime} that $\gamma^{(n)}(I_2^{(n)\prime})$ is the linear interpolation of the points, in order, $a_1^{(n)}$, ..., $a_{2n+1}^{(n)}$.  Therefore, it follows from \eqref{shiftinga2} that
\begin{equation*}
F(\gamma^{(n)}(I_2^{(n)})) = \exp(i\phi)\gamma^{(n)}(I_2^{(n)}) - \eta = \gamma^{(n)}(I_1^{(n)\prime}).
\end{equation*}
We have thus proven \eqref{proofcondition}, from which it follows that
\begin{equation*}
\begin{aligned}
F(\gamma^{(n)}(I)) 	&= F(\gamma^{(n)}(I_0^{(n)})) \cup F(\gamma^{(n)}(I_1^{(n)})) \cup F(\gamma^{(n)}(I_2^{(n)})) \cup F(\gamma^{(n)}(I_3^{(n)})), \\
				&= \gamma^{(n)}(I_3^{(n)\prime}) \cup \gamma^{(n)}(I_2^{(n)\prime}) \cup \gamma^{(n)}(I_1^{(n)\prime}) \cup \gamma^{(n)}(I_0^{(n)\prime}) \\
				&= \gamma^{(n)}(I),
\end{aligned}
\end{equation*}
up to a set of zero one-dimensional Lebesgue measure. It is clear from the definitions in \eqref{gammaprime} that
\begin{equation*}
\begin{aligned}
I_0^{(n)\prime} &= I_3^{(n)} - t_{4n+2}^{(n)} = I_3^{(n)} - (l^{(n)} + (4n+1)\eta), \\
I_1^{(n)\prime} &= I_2^{(n)} - t_{2n+2}^{(n)} + t_1^{(n)} = I_2^{(n)} - (2n+1)\eta, \\
I_2^{(n)\prime} &= I_1^{(n)} - t_2^{(n)} + t_{2n+1}^{(n)} = I_1^{(n)} + 2n\eta, \\
I_3^{(n)\prime} &= I_0^{(n)} + t_{4n+1}^{(n)} = I_0^{(n)} + l^{(n)} + 4n\eta.
\end{aligned}
\end{equation*}
Hence, \eqref{proofcondition} implies that the mapping $(\gamma^{(n)})^{-1} \circ F \circ \gamma^{(n)} : I^{(n)} \rightarrow I^{(n)}$ is an IET of four subintervals with combinatorial data
\begin{equation}
\label{IET}
\begin{aligned}
\sigma 	&= 	(0~ 3)(1~ 2)\\
\xi^{(n)} 	&= (l^{(n)} + \eta, 2n\eta, 2n\eta, l^{(n)}).
\end{aligned}
\end{equation}
Note that the superscript $n$ here is not to be confused with Rauzy induction.  To prove that $Y^{(n)}$ is $F$-invariant, it suffices to show that each atom in $\mathcal{Q}^{(n)} = \calp \vee Y^{(n)} = \{ P \cap Y^{(n)} : P \in \calp \}$ intersects the others at most on its boundary, that their image under $F$ is contained in $Y^{(n)}$.
Recall the notation in \eqref{polygon} for denoting open polygons in $\C$.  Note that if $E: \C \rightarrow \C$ is any Euclidean isometry on $\C$, then
\begin{equation*}
\begin{aligned}
E([x_0, \ldots, x_{N-1}]) = [E(x_0), \ldots, E(x_{N-1})].
\end{aligned}
\end{equation*}
For $j = 0,1,2,3$, let 
\begin{equation}
\label{Qpartition}
Q_j^{(n)} = Y^{(n)} \cap P_j.
\end{equation}
By the definition of $Y^{(n)}$ in \eqref{Yregion} and the properties \eqref{a2}, \eqref{a2n} and \eqref{a4n} of the points $a_2^{(n)}$, $a_{2n}^{(n)}$ and $a_{4n+2}^{(n)}$ respectively, we have
\begin{equation*}
\begin{aligned}
Q_0^{(n)} 	&= [0, a_{4n+2}^{(n)}, a_{4n+3}^{(n)}, \lambda], \\
Q_1^{(n)} 	&= [0, a_{2n+2}^{(n)}, a_{2n+3}^{(n)}, \ldots, a_{4n+2}^{(n)}], \\
Q_2^{(n)} 	&= [0, a_2^{(n)}, a_3^{(n)}, \ldots, a_{2n+2}^{(n)}], \\
Q_3^{(n)}	&= [-1, a_0^{(n)}, a_2^{(n)}, 0].
\end{aligned}
\end{equation*}
Since $Q_j \subset P_j$ for each $j=0,1,2,3$, we have
\begin{equation}
\label{Q03}
\begin{aligned}
F(Q_0^{(n)}) 	&= [0, a_{4n+2}^{(n)}, a_{4n+3}^{(n)}, \lambda] - 1 = [-1, a_0^{(n)}, a_1^{(n)}, -\eta], \\
F(Q_3^{(n)}) 	&= [-1, a_0^{(n)}, a_2^{(n)}, 0] + \lambda = [-\eta, a_{4n+1}^{(n)}, a_{4n+3}^{(n)}, \lambda].
\end{aligned}
\end{equation}
Recall \eqref{shiftinga1}, then it is clear to see that
\begin{equation}
\label{Q1}
\begin{aligned}
F(Q_1^{(n)}) 	&= \exp(i\phi)[0, a_{2n+2}^{(n)}, \ldots, a_{4n+2}^{(n)}] - \eta \\
		&= [-\eta, \exp(i\phi)a_{2n+2}^{(n)} - \eta, \ldots, \exp(i\phi)a_{4n+2}^{(n)} - \eta] \\
		&= [-\eta, a_1^{(n)}, \ldots, a_{2n+1}^{(n)}].
\end{aligned}
\end{equation}
Similarly by recalling \eqref{shiftinga2} we get
\begin{equation}
\label{Q2}
\begin{aligned}
F(Q_2^{(n)}) 	&= \exp(-i\phi)[0, a_2^{(n)}, \ldots, a_{2n+2}^{(n)}] - \eta \\
		&= [-\eta, \exp(-i\phi)a_2^{(n)} - \eta, \ldots, \exp(-i\phi)a_{2n+2}^{(n)} - \eta] \\
		&= [-\eta, a_{2n+1}^{(n)}, \ldots, a_{4n+1}^{(n)}].
\end{aligned}
\end{equation}
It is clear from \eqref{Q03}, \eqref{Q1}, \eqref{Q2} that $F(Y^{(n)}) \subset Y^{(n)}$, and the only nonempty intersections are
\begin{equation*}
\begin{aligned}
\ol{F(Q_0^{(n)})} \cap \ol{F(Q_1^{(n)})} 	&= \{ -\eta t + (1-t)a_1^{(n)} : t \in [0,1] \}, \\
\ol{F(Q_1^{(n)})} \cap \ol{F(Q_2^{(n)})} 	&= \{ -\eta t + (1-t)a_{2n+1}^{(n)} : t \in [0,1] \}, \\
\ol{F(Q_2^{(n)})} \cap \ol{F(Q_3^{(n)})} 	&= \{ -\eta t + (1-t)a_{4n+1}^{(n)} : t \in [0,1] \},
\end{aligned}
\end{equation*}
and $\ol{F(Q_0^{(n)})} \cap \ol{F(Q_1^{(n)})} \cap \ol{F(Q_2^{(n)})} \cap \ol{F(Q_3^{(n)})} = \{-\eta\}$.  All of these intersections are on the boundaries of the atoms, and thus have 0 two-dimensional Lebesgue measure.  Therefore $Y^{(n)}$ is $F$-invariant up to zero measure.
\end{proof}

\begin{remark}
Note that when $n = 0$, we have $l^{(0)} = \lambda$, and the sequence $(a_k^{(0)})_{k=0}^3$ consists only of the four points
\begin{equation*}
a_0^{(0)} = -1,~a_1^{(0)} = -\eta,~a_2^{(0)} = 0,~a_3^{(0)} = \lambda.
\end{equation*}
The curve $\gamma^{(0)}$ is therefore simply the baseline interval $[-1,\lambda] = I^{(0)} - 1 = \gamma(I^{(0)})$, the action of $T$ on $[-1,\lambda]$ is almost-everywhere equal to the baseline IET, and the region $Y^{(0)}$ is the empty set (it is open with empty interior).

\end{remark}

When $N \geq 2$, this result has the following consequence.

\begin{corollary}[Invariant Layers]
\label{Invariant Layers}
Suppose that $N(\lambda,\phi) \geq 1$.  Then $Y^{(n)} \setminus Y^{(n-1)}$ is  $F$-invariant for all $1 \leq n \leq N$.
\end{corollary}

\begin{proof}
If $n = 1$, then $Y^{(1)} \setminus Y^{(0)} = Y^{(1)}$, which is already $F$-invariant by Theorem \ref{invariantcurves}.  So suppose $N \geq 2$ and $n \geq 2$.  Let $\operatorname{Leb}_2$ denote two-dimensional Lebesgue measure.
Since $l^{(n-1)} > l^{(n)} > 0$, by Theorem \ref{invariantcurves}, $Y^{(n)}$ and $Y^{(n-1)}$ are $F$-invariant.  Thus, $F |_{Y^{(n)}}$ and $F |_{Y^{(n-1)}}$ are invertible, and we have that 
\begin{equation*}
Y^{(n-1)} = F |_{Y^{(n)}}^{-1}(Y^{(n-1)}).
\end{equation*}
up to a set of zero $\operatorname{Leb}_2$-measure.  Thus, we know that
\begin{equation*}
\operatorname{Leb}_2\left( \left( Y^{(n)} \setminus Y^{(n-1)} \right) \cap F |_{Y^{(n)}}^{-1}(Y^{(n-1)}) \right) = 0.
\end{equation*}
By the invariance of $Y^{(n)}$, this implies
\begin{equation*}
\operatorname{Leb}_2 \left( (Y^{(n)} \setminus Y^{(n-1)}) \cap F |_{Y^{(n)}}^{-1}(Y^{(n)} \setminus Y^{(n-1)}) \right) = \operatorname{Leb}_2\left( Y^{(n)} \setminus Y^{(n-1)} \right),
\end{equation*}
but $F |_{Y^{(n)}}$ is area-preserving since it is invertible, so in fact
\begin{equation*}
Y^{(n)} \setminus Y^{(n-1)} = F |_{Y^{(n)}}^{-1}(Y^{(n)} \setminus Y^{(n-1)}),
\end{equation*}
up to a set with Lebesgue measure zero.
\end{proof}

\subsection{Dynamics on the Invariant Curves} 

For each integer $0 \leq n \leq N(\lambda, \phi)$, recall that $I^{(n)} = \left[ 0, 2l^{(n)} + (4n+1)\eta \right)$, and from Theorem \ref{invariantcurves} recall that there is a 4-IET $(\mathcal{I}^{(n)},f^{(n)})$ defined by
\begin{equation}
\label{IETinterval}
\begin{aligned}
\mathcal{I}^{(n)} 	&= \left. 	\begin{cases} 
									I_0^{(n)} &\supset (0, l^{(n)} + \eta), \\
									I_1^{(n)} &\supset (l^{(n)} + \eta, l^{(n)} + (2n+1)\eta), \\
									I_2^{(n)} &\supset (l^{(n)} + (2n+1)\eta, l^{(n)} + (4n+1)\eta), \\
									I_3^{(n)} &\supset (l^{(n)} + (4n+1)\eta, 2l^{(n)} + (4n+1)\eta)
								\end{cases} \right\}, \\
f^{(n)}(x) 	        &=	 \begin{cases}
								x + l^{(n)} + 4n\eta 		&\text{ if } x \in I_0^{(n)}, \\
								x + (2n-1)\eta 			&\text{ if } x \in I_1^{(n)}, \\
								x - (2n+1)\eta 			&\text{ if } x \in I_2^{(n)}, \\
								x - l^{(n)} - (4n+1)\eta	&\text{ if } x \in I_3^{(n)}. \\
								\end{cases}
\end{aligned}
\end{equation}
given by the combinatorial data $(\sigma,\xi^{(n)})$ from \eqref{IET}.
\begin{theorem}
\label{densitythm}
Either the $F$-orbits of all points on $\gamma^{(n)}(I^{(n)})$ are dense in $\gamma^{(n)}(I^{(n)})$ or they are periodic.  Furthermore, $F$-orbits along $\gamma^{(n)}(I^{(n)})$ are dense if and only if
\begin{equation*}
\lambda \neq \frac{p + 2n(1 + \cos\phi)q}{p + (1 + 2n(1 + \cos\phi))q},
\end{equation*}
for any $p,q \in \Z$ with $q > 0$.
\end{theorem}
\begin{proof}
By Theorem \ref{invariantcurves}, we have that for all $x \in I^{(n)}$,
\begin{equation*}
F \circ \gamma^{(n)}(x) = \gamma^{(n)} \circ f^{(n)}(x).
\end{equation*}
The first return map of $f^{(n)}$ to the first subinterval $I_0^{(n)}$ is a 2-IET $f_0^{(n)}$ with data
\begin{equation*}
\begin{aligned}
\zeta 	&= 	(0~1), \\
\nu^{(n)} 	&= 	\left( l^{(n)}, \eta \right).
\end{aligned}
\end{equation*}
It is not so difficult to show that under $f_0^{(n)}$, the subinterval $I_0^{(n)}$ visits all of $I^{(n)}$ in $4n$ iterations, and that the dynamics of $f^{(n)}$ is determined by the first return map $f_0^{(n)}$.  It is known that 2-IETs are conjugate to circle rotations. Hence every orbit is either periodic or every orbit is dense.  Moreover, orbits under $f^{(n)}$ are dense in $I^{(n)}$ if and only if they are dense in $I_0^{(n)}$ under $f_0^{(n)}$.  

Recall that $\eta > 0$, so orbits are periodic on the curve if and only if
\begin{equation*}
\frac{l^{(n)}}{\eta} \in \Q.
\end{equation*}
Because $l^{(n)} = \lambda - 2n\eta(1 + \cos\phi)$, this is equivalent to
\begin{equation*}
\frac{\lambda}{\eta} \in 2n \cos\phi + \Q.
\end{equation*}
Further recalling that $\eta = 1 - \lambda$, we see that this is equivalent to the statement that there exist $p,q \in \Z$, $q > 0$ such that
\begin{equation*}
\frac{1 - \lambda}{\lambda} = \frac{q}{2nq\cos\phi + p}.
\end{equation*}
Adding 1 to both sides reciprocating gives us
\begin{equation*}
\lambda = \frac{p + 2n(1 + \cos\phi)q}{p + (1 + 2n(1 + \cos\phi))q}.
\end{equation*}
Conversely, $f^{(n)}$-orbits are dense on $I^{(n)}$ if and only if
\begin{equation*}
\lambda \neq \frac{p + 2n(1 + \cos\phi)q}{p + (1 + 2n(1 + \cos\phi))q},
\end{equation*}
for every $p,q \in \Z$ such that $q > 0$.
\end{proof}

Observe that the IET given by data $(\sigma,\xi^{(n)})$ as in \eqref{IET} does not satisfy the Keane condition.  In particular, recall from \eqref{IETinterval} that 
\begin{equation*}
\partial I_1^{(n)} = l^{(n)} + \eta \text{ and } \partial I_2^{(n)} = l^{(n)} + (2n+1)\eta.
\end{equation*}
where $\partial [a,b) = a$.  Then, by using the formula for $f^{(n)}$ in \eqref{IETinterval}, a simple inductive argument shows that for $0 \leq k \leq n-1$, we have
\begin{align*}
(f^{(n)})^{2k+1}(l^{(n)}+\eta)    &= (f^{(n)})^{2k}(l^{(n)}+2n\eta) \\
                                                    &= l^{(n)} + (2n - 2k)\eta.
\end{align*}
Specifically, when $k = n-1$ we get
\begin{equation*}
(f^{(n)})^{2n-1}(l^{(n)}+\eta) = l^{(n)} + 2\eta,
\end{equation*}
and finally,
\begin{equation*}
(f^{(n)})^{2n}(l^{(n)}+\eta) = l^{(n)} +2\eta + (2n-1)\eta = l^{(n)} + (2n+1)\eta = \partial I_2^{(n)}.
\end{equation*}

\section{Discussion}   

The existence of a new concrete example of a family of TCEs that have invariant curves that are embeddings of IETs makes Theorem \ref{invariantcurves} the most important result of this paper, particularly as through Corollary \ref{Invariant Layers} we conclude that these curves have the effect of restricting the distribution of orbits in the phase space.  In particular,  the invariance of $Y^{(n)}$ as in Theorem \ref{invariantcurves} and the invariance of the layers $Y^{(k)} \setminus Y^{(k-1)}$ as per Corollary \ref{Invariant Layers} suggest that if the exceptional set intersects the interior of more than one layer, it cannot support an ergodic invariant measure. However, determining the set of parameters such that the exceptional set intersects in this way remains an open problem. Although the curves we study here are trivial embeddings in the sense of \cite{AGPR}, they are nonetheless interesting because they divide the phase space. They also potentially give information about nontrivial embeddings that may accumulate on them.

For such polygonal embeddings of IETs into PWIs, there is a connection between the renormalization structure of the IET and the geometry of its embedding, namely that the piecewise linear nature of some IET-embeddings could arise from a failure for the underlying IET to meet the Keane condition, and that the number of pieces in a piecewise linear embedding roughly corresponds to the number of iterates it takes for two interval boundaries to coincide.  This leads us to conjecture that when an IET-embedding does not consist of line segments or circle arcs, the underlying IET must satisfy the Keane condition.

Recall the PWI defined by $(\mathcal{Q}^{(n)},F |_{Y^{(n)}})$, where $\mathcal{Q}^{(n)}$ is as in \eqref{Qpartition}.  Despite the appearance of non-convex atoms, we conjecture that the non-convex atom breaks down into convex cells after a finite number of dynamical refinements.  This appears to be the case because the edges that ``produce'' this non-convexity are and remain on the boundary of $Y$ for all iterates.  See figures \ref{restriction_to_Y} 
for example.  

\begin{figure}
\centering
\includegraphics[width=\linewidth]{{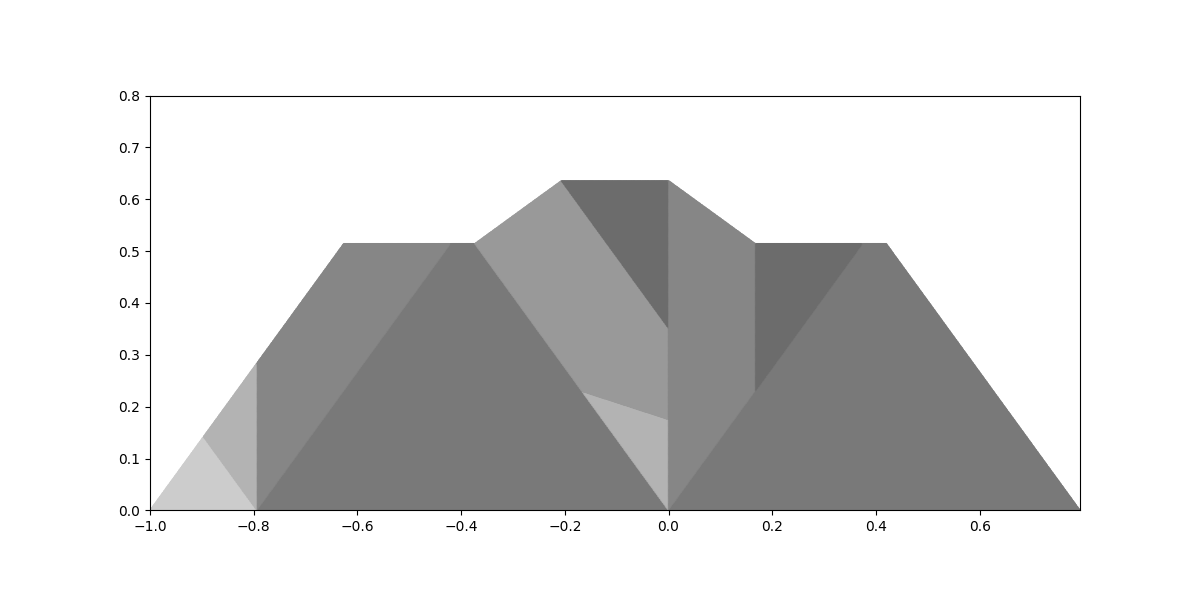}}
\caption{An illustration of the first dynamical refinement of the partition $\mathcal{Q}^{(1)}$ under the map $F |_{Y^{(1)}}$, where $\phi = \pi/5$, $\rho = 1$ and $\lambda = \frac{2+2\cos\phi}{3+2\cos\phi} + 0.01$.  Observe that the non-convex atom appears subdivided into convex pieces, meaning that the non-convexity is not pathological in this case, in that the encoding associated to each $n$-cell is unique.}
\label{restriction_to_Y}
\end{figure}


\begin{figure}
\centering
\includegraphics[width=\linewidth]{{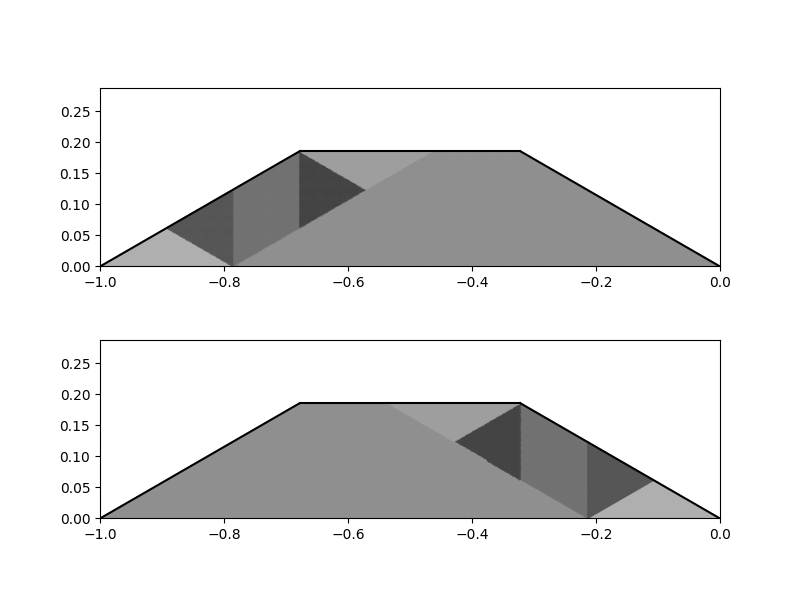}}
\caption{An illustration of the first return map of a TCE of the form \eqref{parameters}, with parameters $\phi = \pi/4$ and $\lambda = \pi/4$, to the trapezium $Q_3^{(1)} = Y^{(1)} \cap P_3$.  Observe that this appears to be an invertible PWI on a compact domain for which there are two distinct IETs on the top and bottom edges of the boundary. This is a simple example, but it is not true in general that this map will have a finite partition.}
\label{trapezium1}
\end{figure}

Consider the first return map $F$ of $(\mathcal{Q}^{(n)}, F |_{Y^{(n)}})$ to the trapezium atom $Q_3^{(n)} = P_3 \cap Y^{(n)}$.  This mapping is a PWI, though the partition seems to depend on the angle, and it is not guaranteed that the partition is finite.  In all cases, however, the mapping has the property that its action on the top and bottom edges of $Q_3^{(n)}$ are two distinct 2-IETs. Thus the mapping on the trapezium $Q_3^{(n)}$ could be viewed as a transition map between these two IETs, though this map is unstudied.  See figure \ref{trapezium1} for an example of such a mapping on the trapezium.  In a similar way, as discussed in Section 3, PWIs with $d=1$, $\alpha \in \mathbb{B}^3$, $0 < \lambda < 1$ and $\sigma$ the identity permutation also induce mappings on a trapezium between the IET on the baseline and its first Rauzy induction.  A natural question is whether these PWIs are uniquely defined by the 2-IETs on their top and bottom edges.  Broadly, let $f_1$ and $f_2$ be IETs on the intervals $I_1 = [0,l_1]$ and $I_2 = [a,a+l_2]$, respectively, for $a \in \R$, $l_1,l_2 > 0$. Set $h > 0$, and let $X$ denote the closed trapezium defined by the vertices $(0,0)$, $(l_1,0)$, $(a,h)$, and $(a+l_2,h)$.  Then we ask how many PWIs $(\calp,F)$ exist on $X$ which satisfy the boundary conditions
\begin{equation*}
F(x,0) = (f_1(x),0) \text{ for all } 0 \leq x < l_1, \\
\end{equation*}
and
\begin{equation*}
F(x,h) = (f_2(x),h) \text{ for all } a \leq x < a+l_2?
\end{equation*}

\subsection*{Funding Acknowledgement}
NC was supported by a studentship from the UK Engineering and Physical Sciences Research Council (EPSRC).

\subsection*{Data Access Statement}
No new data were generated or analysed during this study.  The figures in this study were produced by Python programming code written by NC.  This code is publicly available at:
{\tt https://github.com/NoahCockram/pyTCE/tree/main}. For the purpose of open access, the authors have applied a Creative Commons Attribution (CC BY) licence to any Author Accepted Manuscript version arising from this submission.

\end{document}